\documentclass[12pt,a4paper, english]{article}
\usepackage[top=1in,bottom=1in,left=1in,right=1in]{geometry}
\usepackage{epsfig,amsmath,amsthm,amssymb,latexsym,graphicx,setspace,sectsty}
\usepackage{enumerate,url,natbib,booktabs,xcolor,bbm}
\usepackage[T1]{fontenc}
\usepackage{babel}
\RequirePackage[colorlinks,citecolor=blue,urlcolor=blue]{hyperref}

\usepackage{float} 
\floatstyle{plain}
\newfloat{Algorithm}{thp}{lop}
\floatname{Algorithm}{Algorithm}

\newcommand\diag{\textnormal{diag}}
\newcommand\tr{\textnormal{tr}}
\newcommand\var{\textnormal{var}}

\newcommand\e{\textnormal{e}}
\newcommand\opt{\textnormal{opt}}

\newcommand\bone{{\bf{1}}}
\newcommand\rmm{r_-}
\newcommand\rpp{r_+}
\newcommand\vp{\varphi_+}
\newcommand\vm{\varphi_-}

\newtheorem{theorem}{Theorem}
\newtheorem{lemma}{Lemma}
\newtheorem{corollary}{Corollary}
\newtheorem{property}{Property}
\newtheorem{definition}{Definition}
\newtheorem{remark}{Remark}

\graphicspath{{img/}}

\author{
Linda S. L. Tan\\
\normalsize Department of Statistics and Applied Probability,\\ 
\normalsize National University of Singapore, \\
\normalsize Blk S16, Level 7, 6 Science Drive 2, Singapore 117546\\
\normalsize statsll@nus.edu.sg
}
\date{}
\title{Explicit inverse of tridiagonal matrix with \\
applications in autoregressive modeling}

\begin{document}
\maketitle

\begin{abstract}
We present the explicit inverse of a class of symmetric tridiagonal matrices which is almost Toeplitz, except that the first and last diagonal elements are different from the rest. This class of tridiagonal matrices are of special interest in complex statistical models which uses the first order autoregression to induce dependence in the covariance structure, for instance, in econometrics or spatial modeling. They also arise in interpolation problems using the cubic spline. We show that the inverse can be expressed as a linear combination of Chebyshev polynomials of the second kind and present results on the properties of the inverse, such as bounds on the row sums, the trace of the inverse and its square, and their limits as the order of the matrix increases.
\\[2mm]
Mathematical sub classification numbers: 15A09, 65F50, 62M10, 91G70, 91B72
\end{abstract}

\section{Introduction}
Tridiagonal matrices occur in many areas of science, such as mathematics, econometrics and quantum mechanics. For instance, tridiagonal linear systems often arise in the solving of interpolation problems, boundary value problems and partial differential equations using finite difference methods \citep{Pozrikidis2014}. The inversion of both the general form as well as some special classes of tridiagonal matrices has thus been studied extensively \citep{Schlegel1970, Lewis1982, Heinig1984}. A comprehensive review is given in \cite{Meurant1992}. Various algorithms have also been proposed for efficient computation of the inverse \citep{ElMikkawy2006, Hadj2008, Ran2009}. 

Formulas for the inverse of the general tridiagonal matrix have been derived by several authors based on different approaches \citep[e.g.][]{Yamamoto1979, Usmani1994, Huang1997}, such as linear difference equation \citep{Mallik2001} and backward continued fractions \citep{Kilic2008}. These formulas usually involve recurrence relations and are not of explicit closed form, or they may reduce to closed form only for some special classes. \cite{Meurant1992} presents an explicit inverse for the Toeplitz tridiagonal matrix by solving the recurrences in its Cholesky decomposition analytically. Extending these results, \cite{Fonseca2001, Fonseca2005} express the inverse of $k$-Toeplitz tridiagonal matrices explicitly in terms of Chebyshev polynomials of the second kind. An order $n$ $k$-Toeplitz tridiagonal matrix is of the form $A=[a_{ij}]$ with $a_{i+k, j+k} = a_{ij}$ for $i,j=1, \dots, n-k$ and $a_{ij} = 0$ if $|i-j|>1$. That is, the diagonal, subdiagonal and superdiagonal entries are $k$-periodic. A Toeplitz tridiagonal matrix is obtained when $k=1$. \cite{Encinas2018a} present the explicit inverse of a $(p, r)$-Toeplitz tridiagonal matrix, in which each diagonal is a quasi-periodic sequence with period $p$ but multiplied by a real number $r$. Such analytic formulas are important in studying the properties of the inverse, for instance, the rate of decay of elements along a row or column or for establishing bounds \citep{Nabben1999a, Nabben1999b}. In a closely related but independent work, \cite{Encinas2018b} present the explicit inverse of tridiagonal matrices, and necessary and sufficient conditions for their invertibility, which are derived using the solution of Sturm-Liouville boundary value problems associated to second order linear difference equations expressed through a discrete Schr\"{o}dinger operator. Our article focuses on the inverse of a more specialized form of the tridiagonal matrix and its properties, which has applications in econometric and statistical modeling problems.

In this article, we consider real symmetric $n \times n$ tridiagonal matrices of the form 
\begin{equation} \label{tridiag_form}
Q = \begin{bmatrix}
d  & -b & \cdots & 0 & 0 \\
-b & c & \cdots & 0 & 0 \\
\vdots &  \vdots & \ddots & \vdots  & \vdots \\
0 & 0 &  \cdots & c & -b \\
0 &  0 &  \cdots & -b & d \\
\end{bmatrix},
\end{equation}
which is almost Toeplitz except that the first and last diagonal elements are different from the rest. We assume without loss of generality that $b=1$ or $-1$, $c\geq 0$ and $\lambda = c-d \neq 0$. Tridiagonal matrices of this form often arise in interpolation problems, as well as econometrics and spatial modeling when first order autoregression is used to induce dependence in the covariance structure. We note that \cite{Yueh2006}, \cite{Yueh2008} and \cite{Cheng2013} has derived the explicit inverse of complex tridiagonal matrices with constant diagonals and some perturbed elements using symbolic calculus, of which $Q$ arises as a special case. However, here we focus on the real-valued case and present an alternative proof using difference equations (which may be more accessible to practitioners in statistics and econometrics) and show that the elements of $Q^{-1}$ can be viewed as a linear combination of Chebyshev polynomials of the second kind. While \cite{Yueh2006} (Theorems 2 and 3) do consider simplified expressions for some special cases where the perturbed elements are of equal value and at the corners, these cases correspond only to $\lambda = \pm 1$ in the context of $Q$.  We first describe some motivating applications in Section \ref{Sec_appl} for which the results in this article are of special relevance, before deriving the explicit inverse. We also study some properties of $Q^{-1}$. Bounds for the row (or column) sums are also presented and we provide expressions for the trace of $Q^{-1}$ and $Q^{-2}$ and their limiting behavior as the order of the matrix increases. Application of these results are illustrated in Section \ref{sec:eg} using a first order autoregressive process with observational noise.

\section{Applications} \label{Sec_appl}
In this section, we discuss some applications where the tridiagonal matrix in \eqref{tridiag_form} arises, and where the explicit inverse and its properties may be of interest.  

In interpolation problems, the cubic spline is often used to avoid the Runge phenomenon \citep{Knott2000}. When $n$ equidistant knots and clamped (first derivative) boundary conditions are used, a tridiagonal matrix in the form of $Q$ arises in the linear system for evaluating the coefficients of the piecewise cubic polynomials, where $b=-1$, $c=4$ and $d=2$ \citep[see][]{Revesz2014}. 

In econometrics, the stationary first order autoregression, AR(1), is defined as
\begin{equation*}
y_t = \mu + \phi y_{t-1} + \epsilon_t, \quad (t=1, \dots, n)
\end{equation*}
where $\mu$ is a constant, $|\phi|<1$ and $\epsilon_t \sim \text{WN}(0, \sigma_\epsilon^2)$ is a white noise process with zero mean and variance $\sigma_\epsilon^2$. If $\phi \neq 0$, the inverse covariance matrix of $y=(y_1, \dots, y_n)^T$ is $\sigma_\epsilon^{-2} |\phi| Q$, where $b=\phi/|\phi|$, $c=(1+\phi^2)/|\phi| > 2$ and $d=1/|\phi| > 1$. The covariance matrix of the AR(1) and hence the closed form of $Q^{-1}$ is well-known \citep{Nerlove1979}. However, the AR(1) is often used as building blocks in more complex models  as a means of introducing dependence, and deriving the covariance structure in such models becomes a more challenging task. 

Consider for instance the conditional autogressive (CAR) model, which is widely used in modeling spatial data \citep{Besag1974, Cressie1993, Wall2004}. Let $y_i$ denote the observation at site $i$ for $i=1, \dots, n$ and $y_{-i} = \{y_j:j \neq i\}$. The CAR model incorporates spatial dependence into the covariance structure by specifying a Gaussian distribution for
\begin{equation*}
y_i|y_{-i} \sim N\left( \mu_i + \sum\nolimits_{j=1}^n c_{ij} (y_j - \mu_i), \sigma_i^2 \right),
\end{equation*}
where $\mu_i$ is the mean of $y_i$, $\sigma_i^2$ is the conditional variance and $c_{ij}$ is a covariance parameter. Let $\mu=(\mu_1, \dots, \mu_n)$, $C=[c_{ij}]$ and $T=\diag(\sigma_1^2, \dots, \sigma_n^2)$. Then $y=(y_1, \dots, y_n)^T \sim N(\mu, \Sigma)$, where
\begin{equation*}
\Sigma = (I_n - C)^{-1}T.
\end{equation*}
It is common to write $C= \rho W$, where $\rho$ is a spatial dependence parameter and $W$ represents the neighborhood structure of the $n$ sites. If the adjacency structure of a path graph is adopted and the rows of $W$ are restricted to sum to 1, then $W$ is a tridiagonal matrix with a zero diagonal, superdiagonal $(1, 0.5, \dots, 0.5)$ and subdiagonal $(0.5, \dots, 0.5,1)$. For the CAR model, $T^{-1}C$ must be symmetric, which implies that $T =\sigma^2 \text{diag}(1, 0.5, \dots, 0.5, 1)$ for some $\sigma^2>0$. In addition, $|\rho| < 1$ for $\Sigma$ to be positive definite. If $\rho \neq 0$, $\Sigma^{-1} = T^{-1} - \rho T^{-1}W = \sigma^{-2} |\rho| Q$ where $b=\rho/|\rho|$, $c=2/|\rho| > 2$ and $d=1/|\rho| > 1$. \cite{Muenz2017} notes that the elements of $\Sigma$ and hence $Q^{-1}$ are of interest for model calibration.

In Section \ref{sec:eg}, we discuss a reparametrization of the AR(1) plus noise model and use this setting demonstrate how the results derived in this article can be applied in econometrics and statistical modeling.

\section{Inverse of tridiagonal matrix}
First we introduce Chebyshev polynomials of the second kind, which are defined as solutions to the recurrence equation
\begin{equation} \label{Chebyshev_recurrence}
U_{n+1}(x) - 2x U_n(x) + U_{n-1}(x) = 0, \quad n=1, 2, \dots
\end{equation}
with initial conditions $U_0(x) = 1$ and $U_1(x) =2x$. The solution is of the form
\begin{equation*}
U_n(x) =
\begin{cases}
{\sin\{(n+1)\theta\}}/{\sin \theta}, \text{ where } \cos\theta =x  & \text{if } |x| < 1, \\
(\pm1)^n (n+1) & \text{if } x=\pm1, \\
{\sinh\{(n+1)\theta\}}/{\sinh\theta}, \text{ where } \cosh\theta =x  & \text{if } |x| > 1.
\end{cases}
\end{equation*}
The definition of Chebyshev polynomials of the second kind may be extended to negative indices via $U_{-n}(x) = -U_{n-2}(x)$ for negative $n$, with $U_{-1}(x) = 0$.

In Lemma \ref{lemma_de}, we show that the solution of a second order difference equation, which differs from \eqref{Chebyshev_recurrence} in terms of initial conditions, can be expressed as a linear combination of $U_n(c/2)$. Then we present a closed form expression for $Q^{-1}$ in Theorem \ref{thm_invOmega}.

\begin{definition} \label{def1}
If $c \neq 2$, then $r_{\pm} = (c \pm \sqrt{c^2-4})/2$, $\varphi_{\pm} = r_{\pm} - \lambda$, $\kappa_i = \vp\rpp^i - \vm \rmm^i$ where $i$ is an integer and $\kappa = \vp^2 \rpp^{n-1} - \vm^2 \rmm^{n-1}$.
\end{definition}
\noindent Note that $\rpp \rmm = 1$, $\rpp + \rmm =c$ and $\kappa_0 = \rpp - \rmm$. In addition, $\kappa_i$ and $\kappa$ are real if $c>2$ and purely imaginary if $c<2$.

\begin{lemma} \label{lemma_de}
Consider the second order difference equation
\begin{equation} \label{recur_rel}
\beta_i = c\beta_{i-1} - \beta_{i-2}, \quad (i=2, \dots, n-1),
\end{equation}
with initial conditions $\beta_0 = 1$ and $\beta_1 = d$. Let $ \lambda = c-d \neq 0$. The solution is given by
\begin{equation} \label{beta_i}
\beta_i = U_i(c/2) - \lambda U_{i-1}(c/2),  \quad (i=0, 1,\dots, n-1).
\end{equation}
If $c \neq 2$, then $\beta_i=\kappa_i/\kappa_0$ for $i=0, 1,\dots, n-1$.
\end{lemma}
\begin{proof}
The characteristic equation of the second-order difference equation is $r^2-cr+1=0$ and the roots are $r_{\pm} = (c \pm \sqrt{c^2-4})/2$.

If $c =2$, the characteristic equation has a single root $r=1$ and the general solution  is $\beta_i = d_0+d_1i$. Applying the initial conditions, we obtain $d_0=1$ and $d_1=d-1$. Thus $\beta_i = 1 +(d-1)i = 1+(1-\lambda)i = U_i(1) - \lambda U_{i-1}(1)$ for $i=0,1, \dots, n-1$. 

If $c\neq2$, there are two distinct roots and the general solution is $\beta_i = d_0 \rpp^{i}+ d_1 \rmm^{i}$. Since $\rpp+\rmm = c = d + \lambda$, we have from the initial conditions, 
\begin{equation*}
\begin{aligned}
d_0 = \frac{\rpp - \lambda}{\rpp - \rmm}, \quad 
d_1 = - \frac{\rmm - \lambda}{\rpp - \rmm},
\end{aligned}
\end{equation*}
Thus, for $ i=0, 1, \dots, n-1$, 
\begin{equation*}
\beta_i = \frac{\rpp^{i+1} - \rmm^{i+1} - \lambda(\rpp^i - \rmm^i)}{\rpp - \rmm} = \frac{\rpp^{i+1} - \rpp^{-(i+1)} - \lambda(\rpp^i - \rpp^{-i})}{\rpp - \rpp^{-1}}.
\end{equation*}
If $c>2$, the roots are real. Let $x = \cosh^{-1}(c/2) = \ln(c/2+\sqrt{c^2/4 -1}) = \ln(\rpp)$. Then $\cosh(x) = c/2$ and $\rpp = \e^x$ so that
\begin{equation*}
\beta_i = \frac{\sinh\{(i+1)x\} - \lambda \sinh(ix)}{\sinh(x)} = U_i(c/2) - \lambda U_{i-1}(c/2).
\end{equation*}
If $0 \leq c < 2$, the roots are complex and $\rpp = e^{\mathrm{i} \theta} = \cos \theta + \mathrm{i} \sin \theta$ where $\cos\theta = c/2$. Thus 
 \begin{equation*}
\beta_i = \frac{ 2\mathrm{i}\sin\{(i+1)\theta\} - \lambda2\mathrm{i}\sin(i\theta)}{2\mathrm{i}\sin(\theta)} = U_i(c/2) - \lambda U_{i-1}(c/2).
\end{equation*}
Since $d_0 = \vp/\kappa_0$ and $d_1 = \vm/\kappa_0$, it follows from the general solution that $\beta_i = \kappa_i/\kappa_0$. 
\end{proof}

The proof of the result in \eqref{beta_i} can be simplified by noting that $\beta_i = U_i(c/2)$ is a primary solution of \eqref{recur_rel}, and hence \eqref{beta_i} is also a solution of \eqref{recur_rel}. See Theorem 3.1 of \cite{Aharonov2005} and \cite{Encinas2018c}. Thus it suffices to show that \eqref{beta_i} satisfies the initial conditions $\beta_0 = 1$ and $\beta_1 = d$. However, we have presented a constructive proof based on the characteristic equation of the second order difference equation as we wanted to introduce the terms $r_{\pm}$ and $\kappa_i$, and to show that $\beta_i  = \kappa_i/\kappa_0$ when $c \neq 2$. These results will be important in the rest of the article.

\begin{theorem}\label{thm_invOmega}
Suppose $Q$ is a tridiagonal matrix of form \eqref{tridiag_form} where $b=1$ or $-1$, $c\geq 0$ and $\lambda  = c-d \neq 0$. Then $Q^{-1}$ exists if and only if $d\beta_{n-1} - \beta_{n-2} \neq 0$. If $Q^{-1}$ exists, then it is symmetric and $Q^{-1}_{ij} = u_i v_j$ for $i \leq j$, where 
\begin{equation} \label{Qinv}
v_i = \frac{b^{i-1} \beta_{n-i}}{d\beta_{n-1} - \beta_{n-2}}, \quad
u_i = b^{i-1} \beta_{i-1} \quad (i=1, \dots, n).
\end{equation}
\end{theorem}
\begin{proof}
Suppose $Q^{-1}$ exists. From \cite{Meurant1992}, $Q^{-1}_{ij} = u_i v_j$ for $i \leq j$ 
where 
\begin{equation}\label{delta}
\begin{gathered}
v_i = b^{i-1}(\delta_n \dots \delta_{n-i+1})^{-1}, \quad 
u_i= b^{n-i}(\delta_i \dots \delta_n v_n)^{-1},  \quad (i=1, \dots, n), \\
\delta_1 = d, \quad 
\delta_i = c - {1}/{\delta_{i-1}} \quad (i=2, \dots, n-1), \quad
\delta_n = d - {1}/{\delta_{n-1}}.
\end{gathered}
\end{equation}
Let $\delta_i = \beta_i/\beta_{i-1}$ for $i=1, \dots, n-1$. From \eqref{delta}, we have the recurrence relation in \eqref{recur_rel} with the same initial conditions. From Lemma \ref{lemma_de}, $\beta_i = U_i(c/2) - \lambda U_{i-1}(c/2)$ and
\begin{equation*}
\begin{aligned}
v_i  &= b^{i-1} \left( \frac{d\beta_{n-1} - \beta_{n-2}}{\beta_{n-1} }\frac{\beta_{n-1}}{\beta_{n-2}} \dots \frac{\beta_{n-i+1}}{\beta_{n-i}} \right)^{-1} = \frac{b^{i-1} \beta_{n-i}}{d\beta_{n-1} - \beta_{n-2}} \\
u_i &= b^{n-i} \left(\frac{\beta_i}{\beta_{i-1}} \dots  \frac{d\beta_{n-1} - \beta_{n-2}}{\beta_{n-1} } \frac{b^{n-1}}{d\beta_{n-1} - \beta_{n-2}}{\beta_{n-1}}\right)^{-1} = b^{i-1} \beta_{i-1}.
\end{aligned}
\end{equation*} 
Thus we must have $d\beta_{n-1} - \beta_{n-2} \neq 0$. Conversely, if $d\beta_{n-1} - \beta_{n-2} \neq 0$, then $Q^{-1}$ exists as it is given by \eqref{Qinv}.
\end{proof}

\noindent We present below some corollaries of Theorem \ref{thm_rowsum} and some properties of $\{\kappa_i\}$, $\{v_i\}$, $\{u_i\}$ and $Q^{-1}$ which will be useful later.
\begin{corollary}\label{cor_c2}
If $c=2$, 
\begin{equation*}
Q_{ij}^{-1} = b^{i+j} \frac{\{1+(1-\lambda)(i-1)\}\{1+(1-\lambda)(n-j)\}}{(1-\lambda)\{(1-\lambda)(n-1) + 2\}}, \quad (i \leq j).
\end{equation*}
Thus $Q^{-1}$ does not exist if $\lambda=1$ or $\lambda=(n+1)/(n-1)$.
\end{corollary}
\begin{proof}
From Lemma \ref{thm_invOmega}, if $c=2$, $\beta_i = U_i(1) - \lambda U_{i-1}(1) = 1+(1-\lambda)i$ for $i=0, \dots, n-1$. $d\beta_{n-1} - \beta_{n-2} = (1-\lambda)\{(1-\lambda)(n-1) + 2\}$ and $Q^{-1}$ is undefined if the denominator is zero.
\end{proof}

\begin{corollary}\label{cor_cnot2}
If $c \neq 2$, $Q^{-1}_{ij} = u_i v_j$ for $i \leq j$ where 
\begin{equation*}
v_i = b^{i-1} {\kappa_{n-i}}/{\kappa}, \quad
u_i = b^{i-1} {\kappa_{i-1}}/{\kappa_0} \quad (i=1, \dots, n).
\end{equation*}
\end{corollary}
\begin{proof}
The result follows from Theorem \ref{thm_invOmega} and Lemma \ref{lemma_de} since $\beta_i = \kappa_i/\kappa_0$ and 
\begin{equation*}
\begin{aligned}
\kappa_0 (d\beta_{n-1} - \beta_{n-2}) &= d\kappa_{n-1} - \kappa_{n-2} \\
& = (\rpp + \rmm - \lambda) ( \vp\rpp^{n-1} - \vm \rmm^{n-1}) - ( \vp\rpp^{n-2} - \vm \rmm^{n-2}) \\
&= \vp^2 \rpp^{n-1} - \vm^2 \rmm^{n-1} = \kappa.
\end{aligned}
\end{equation*}
\\[-13mm]
\end{proof}

\begin{corollary} \label{cor3}
For $i=1, \dots, n$, $u_iv_n = v_{n-i+1}$.
\end{corollary}

\begin{property} \label{kappatve}
If $c > 2$ and $d >1$, the sequence $\{\kappa_i|i=0, \dots, n-1\}$ is positive and strictly increasing. That is, $0 < \kappa_0 < \kappa_1 < \dots < \kappa_{n-1}$. In addition, $\kappa$ is positive.
\end{property}
\begin{proof}
First we show that $\delta_i=\beta_i/\beta_{i-1}>1$ for $i=1, \dots, n-1$. We have $\delta_1 = d > 1$. If $\delta_{i-1} > 1$, then $\delta_i = c - 1/\delta_{i-1} > c-1 > 1$ for $i=2, \dots, n-1$. By induction, $\delta_i > 1$ for $i=1, \dots, n-1$. Since $\delta_i = \beta_i/\beta_{i-1} = \kappa_i/\kappa_{i-1}$, we have $\kappa_i > \kappa_{i-1}$ for $i=1, \dots, n-1$. Note that $\kappa_0 = \rpp - \rmm > 0$ and $\kappa = d \kappa_{n-1} - \kappa_{n-2} > \kappa_{n-1} - \kappa_{n-2} > 0$. 
\end{proof}

\begin{property} \label{uvtve}
If $c > 2$, $d >1$ and $b=1$, 
\begin{enumerate}[(i)]
\item the sequence $\{u_i|i=1, \dots, n\}$ is positive and strictly increasing,
\item the sequence $\{v_i|i=1, \dots, n\}$ is positive and strictly decreasing,
\item all elements of $Q^{-1}$ are positive.
\end{enumerate}
\end{property}
\begin{proof}
If $b=1$, $u_i = \kappa_{i-1}/\kappa_0> 0$ and $v_i = \kappa_{n-i}/\kappa > 0$ for $i=1, \dots, n$. Thus the results (i) and (ii) follow from Property \ref{kappatve}. For (iii), we also have $Q_{ij}^{-1} = u_iv_j > 0$ for $i \leq j$ and $Q^{-1}$ is symmetric.
\end{proof}

\begin{property} \label{kappadiff}
If $c> 2$, then $\kappa_i - \kappa_{i-1} =  \kappa_1-\kappa_0 + (c-2) \sum_{j=1}^{i-1} \kappa_j$ for $i=1, \dots, n-1$. Thus,
\begin{equation*}
\kappa_{n-i-1} +\kappa_{i} < \kappa_{n-i} + \kappa_{i-1}, \quad 
(i=1, \dots, \lceil n/2 \rceil - 1).
\end{equation*} 
\end{property}
\begin{proof}
Using the recurrence relation $\kappa_i = c\kappa_{i-1} - \kappa_{i-2}$ for $i=2, \dots, n-1$, we have
\begin{equation*}
\begin{aligned}
\kappa_i - \kappa_{i-1} &= \kappa_{i-1} - \kappa_{i-2} + (c-2) \kappa_{i-1} \\
&=  \kappa_{i-2} - \kappa_{i-3} + (c-2) (\kappa_{i-1} + \kappa_{i-2}) \\
&= \cdots  = \kappa_1 - \kappa_0 + (c-2) \sum_{j=1}^{i-1} \kappa_j.
\end{aligned}
\end{equation*}
When $i=1$, the above equality holds trivially. Hence, for $i=1, \dots,  \lceil n/2 \rceil - 1$.
\begin{equation*}
\kappa_{n-i} - \kappa_{n-i-1} - ( \kappa_i - \kappa_{i-1} )
= (c-2) \sum_{j=i}^{n-i-1} \kappa_j > 0.
\qedhere
\end{equation*}
\end{proof}

\section{Row sums of the inverse}
In this section, we derive expressions for the row sums of $Q^{-1}$. 

\begin{theorem}\label{thm_rowsum}
Let $s_i$ denote the sum of the $i$th row of $Q^{-1}$. 
\begin{equation*}
s_i = 
\begin{cases}
\dfrac{(n-1)u_i v_i + i v_i + (n-i+1) v_{n-i+1} }{2} & \text{ if } c=2 \text{ and } b=1, \\
\dfrac{ 1-(b-\lambda)(v_i + v_{n-i+1}) }{c-2b} & \text{ if } c=2 \text{ and } b=-1, \text{ or } c \neq 2.
\end{cases}
\end{equation*}
\end{theorem}
\begin{proof}
We have $s_i =  v_i \sum_{j=1}^{i-1}  u_j + u_i \sum_{j=i}^{n}  v_j $. If $c=2$ and $b=1$, 
\begin{equation*}
\begin{aligned}
s_i &= v_i \sum_{j=1}^{i-1} \{1+(1-\lambda)(j-1) \} + \frac{u_i}{d\beta_{n-1} - \beta_{n-2}} \sum_{j=i}^n \{ 1+(1-\lambda)(n-j)\}  \\
&= \frac{v_i (i-1)(\lambda + u_i)}{2} + \frac{u_i (n-i+1)(v_n + v_i)}{2}  \\
&= \{nu_i v_i + v_i(i-u_i)+ (n-i+1)v_{n-i+1} \}/2.
\end{aligned}
\end{equation*}
If $c=2$ and $b=-1$,
\begin{equation*}
\begin{aligned}
s_i &= v_i \sum_{j=1}^{i-1} (-1)^{j-1}\{1+(1-\lambda)(j-1) \} + \frac{u_i}{d\beta_{n-1} - \beta_{n-2}} \sum_{j=i}^n  (-1)^{j-1}\{ 1+(1-\lambda)(n-j)\} \\
&= \frac{v_i(1+\lambda) - 2u_i v_i - v_i(-1)^{i}(1-\lambda)}{4} + \frac{u_iv_n (1+\lambda) + 2u_i v_i + (-1)^{n-i} (1-\lambda) u_i v_n }{4} \\
&= \{1 + (1+\lambda)(v_i + v_{n-i+1}) \}/4.
\end{aligned}
\end{equation*}
Note that $\sum_{j=1}^n (-1)^j j$ can be evaluated using the result:
\begin{equation} \label{alternating}
\sum_{j=1}^n jk^j = \frac{k\{1 - (n+1) k^n + nk^{n+1}\}}{(1-k)^{2}} .
\end{equation}
Finally if $c \neq 2$, 
\begin{equation*}
\begin{aligned}
s_i &=\frac{v_i}{\kappa_{0}} \sum_{j=1}^{i-1}  b^{j-1} \kappa_{j-1}  + \frac{u_i}{\kappa}  \sum_{j=i}^{n}  b^{j-1} \kappa_{n-j} \\
&=\frac{v_i}{\kappa_{0}} \left\{ \vp \sum_{j=1}^{i-1} (b\rpp)^{j-1} - \vm  \sum_{j=1}^{i-1} (b\rmm)^{j-1} \right\} + \frac{u_i}{\kappa} \left\{   \vp \sum_{j=i}^n b^{j-1} \rpp^{n-j} - \vm \sum_{j=i}^n b^{j-1} \rmm^{n-j} \right\}\\
&= \frac{v_i}{\kappa_{0}}  \frac{b^{i}\kappa_{i-2} - b^{i-1} \kappa_{i-1} + (1- b\lambda) \kappa_0}{2-bc} +
\frac{u_i}{\kappa}  \frac{ b^{i+1} \kappa_{n-i} - b^{i} \kappa_{n+1-i} + (1 - b\lambda) b^{n-1}\kappa_0}{2-bc} \\
&= \frac{b (\kappa_{n-1} \kappa_{i-2} - \kappa_{n-i+1} \kappa_{i-1})} {\kappa_0 \kappa (2-bc)}  +  \frac{(1-b\lambda) (v_i + u_i v_n)}{2-bc} \\
&= \frac{-b  +  (1-b\lambda) (v_i + v_{n-i+1})}{2-bc} = \frac{ 1- (b-\lambda)(v_i + v_{n-i+1}) }{c-2b}.
\end{aligned}
\end{equation*}
\\[-12mm]
\end{proof}

\section{Trace of the inverse}
In this section, we derive the trace of $Q^{-1}$ and $Q^{-2}$. We focus on the case where $c \neq 2$. These closed expressions are particularly useful as the order of $Q$ increases. We also present some results on the limiting behavior of $\tr(Q^{-1})$ and $\tr(Q^{-2})$ as $n \rightarrow \infty$. 

\begin{definition}
Let $\rho = \vp \vm = (\rpp-\lambda)(\rmm-\lambda) = 1-c\lambda + \lambda^2$. 
\end{definition}

\begin{theorem} \label{thm_trace}
If $c\neq 2$, 
\begin{equation*}
\tr(Q^{-1})  = {n  (\vp^2\rpp^{n-1} + \vm^2 \rmm^{n-1})  }/{(\kappa_0 \kappa)} - {2\rho  (\rpp^n - \rmm^n)}/{(\kappa_0^2 \kappa)}.
\end{equation*}
\end{theorem}
\begin{proof}
We have $\tr(Q^{-1}) = \sum_{i=1}^n u_iv_i = (\kappa_0 \kappa)^{-1} \sum_{i=1}^n \kappa_{i-1} \kappa_{n-i}$ and 
\begin{equation*}
\begin{aligned}
\sum_{i=1}^n \kappa_{i-1} \kappa_{n-i} 
&= \sum_{i=1}^n  (\vp \rpp^{i-1} - \vm \rmm^{i-1}) (\vp \rpp^{n-i} - \vm\rmm^{n-i}) \\
&= n (\vp^2\rpp^{n-1} + \vm^2 \rmm^{n-1}) - \rho \sum_{i=1}^n ( \rpp^{i-1} \rmm^{n-i} + \rmm^{i-1}\rpp^{n-i}) \\
&= n (\vp^2\rpp^{n-1} + \vm^2 \rmm^{n-1}) - 2\rho \sum_{i=1}^n \rpp^{2i - n -1}. 
\end{aligned}
\end{equation*}
As $\rpp\rmm = 1$, $\sum_{i=1}^n \rpp^{i-1} \rmm^{n-i} = \sum_{i=1}^n \rpp^{2i-n-1}$ and $\sum_{i=1}^n \rmm^{i-1}\rpp^{n-i}  =  \sum_{i=1}^n \rpp^{n -2i+1}$. These two sums are equal as the powers in both cases form an arithmetic progression from $1-n$ to $n-1$ with a difference of 2. Finally,
\begin{equation*}
 \sum_{i=1}^n \rpp^{2i - n -1} = \frac{ \rpp^{1-n}(\rpp^{2n} -1)}{\rpp^2 -1}  = \frac{ \rpp^{n} -\rpp^{-n}}{\rpp - \rmm} = \frac{\rpp^n - \rmm^n}{\kappa_0} .
\qedhere
\end{equation*}
\end{proof}

\begin{remark}
If $c=2$, it is straightforward to show that 
\begin{equation*}
\tr(Q^{-1}) = \frac{n\{ n-n\lambda + \lambda + (1-\lambda)^2(n-1)(n-2)/6\}}{(1-\lambda)\{(1-\lambda)(n-1) + 2\}}.
\end{equation*}
\end{remark}

\begin{corollary} \label{cor4}
If $c > 2$, then $\lim_{n\rightarrow \infty} n^{-1} \tr(Q^{-1}) = \kappa_0^{-1}$.
\end{corollary}
\begin{proof}
If $c > 2$, then $r_{\pm}$ are real and $\rpp > 1$ which implies $0 < \rmm < 1$ since $\rpp\rmm = 1$. Thus $\lim_{n \rightarrow \infty} \rmm^n = 0$. From Theorem \ref{thm_trace}, $\lim_{n\rightarrow \infty} n^{-1} \tr(Q^{-1})$ equals
\begin{equation*}
\begin{aligned}
\frac{1}{\kappa_0}\lim_{n\rightarrow \infty} \frac{1+ \vm^2 \vp^{-2} \rmm^{2n-2}  }{1 - \vm^2 \vp^{-2} \rmm^{2n-2} } - \frac{2\rho}{\kappa_0^2} \lim_{n\rightarrow \infty} \frac{ 1 - \rmm^{2n} }{n\{\vp^2\rmm - \vm^2 \rmm^{2n-1}\}}.
\end{aligned}
\end{equation*}
The left term approaches $\kappa_0^{-1}$, while the right term goes to zero. 
\end{proof}

\begin{lemma}
\begin{equation*}
\zeta_j = \sum_{i=0}^{j-1} \kappa_{i}^2 =  \lambda^2 -1 - 2 \rho j + \frac{\vp^2 \rpp^{2j-1} - \vm^2 \rmm^{2j-1}}{\kappa_0}
\end{equation*}
\end{lemma}
\begin{proof}
Since $\vp \vm = \rho$, $\rpp-\rmm = \kappa_0$ and $\vm^2 \rpp - \vp^2 \rmm = \kappa_0 (\lambda^2-1)$,
\begin{equation*}
\begin{aligned}
\zeta_j &= \sum_{i=0}^{j-1}  (\vp \rpp^i - \vm\rmm^i)^2  = -2\rho j  + \vp^2 \sum_{i=0}^{j-1}  \rpp^{2i} + \vm^2  \sum_{i=0}^{j-1}  \rmm^{2i} \\
& = -2\rho j  +  \{\vp^2(\rpp^{2j-1} - \rmm) - \vm^2 (\rmm^{2j-1} - \rpp) \}/{\kappa_0} \\
& = \lambda^2 -1 -2\rho j  + ({\vp^2 \rpp^{2j-1} - \vm^2 \rmm^{2j-1})}/{\kappa_0}.
\end{aligned}
\end{equation*}
\\[-12mm]
\end{proof}

\begin{theorem} \label{thm_trace2}
If $c\neq 2$, then $\tr(Q^{-2}) = (\kappa_0^2 \kappa^2)^{-1} \mathcal{S}$, where
\begin{equation*}
\begin{aligned}
\mathcal{S} &= 4\rho^2 n^2 -8\rho n (\lambda^2-1) + 4\rho(1 + \lambda^2) + {nc(\vp^4 \rpp^{2n-2} - \vm^4\rmm^{2n-2})}/{\kappa_0} + 2(\lambda^2 -1)^2\\
& \quad  + {16\rho^2}/{\kappa_0^2}
- {4\rho c( \vp^2\rpp^{2n-1} +\vm^2\rmm^{2n-1})}/{\kappa_0^2}  
+ {2(\lambda^2-1)(\vp^2 \rpp^{2n-1} - \vm^2 \rmm^{2n-1} )}/{\kappa_0} .
\end{aligned}
\end{equation*}
\end{theorem}
\begin{proof}
By using the fact that $\tr(Q^{-2})  = \sum_{i=1}^n Q^{-2}_{ii} = \sum_{i=1}^n \sum_{j=1}^n Q_{ij}^{-1} Q_{ji}^{-1}$,
we have
\begin{equation*}
\begin{aligned}
\tr(Q^{-2}) 
= 2 \sum_{j=1}^n \sum_{i=1}^j u_i^2 v_j^2  - \sum_{i=1}^n u_i^2 v_i^2 
= \frac{1}{\kappa_0^2 \kappa^2} \left( 2\sum_{j=1}^n \kappa_{n-j}^2 \zeta_j - \sum_{i=1}^n \kappa_{i-1}^2 \kappa_{n-i}^2 \right).
\end{aligned}
\end{equation*}
We have 
\begin{equation*}
\begin{aligned}
&\sum_{i=1}^n \kappa_{i-1}^2 \kappa_{n-i}^2 = \sum_{i=1}^n  (\vp \rpp^{i-1} - \vm \rmm^{i-1})^2 (\vp \rpp^{n-i} - \vm\rmm^{n-i})^2 \\
&= n(\vp^4 \rpp^{2n-2} + \vm^4 \rmm^{2n-2} + 4\rho^2) - 4\rho \vp^2 \sum_{i=1}^n \rpp^{2i-2} - 4\rho \vm^2 \sum_{i=1}^n \rmm^{2i-2} + 2 \rho^2   \sum_{i=1}^n  \rmm^{2n-4i+2} \\
&= n(\vp^4 \rpp^{2n-2} + \vm^4 \rmm^{2n-2} + 4\rho^2) - \frac{4\rho}{\kappa_0} (\vp^2 \rpp^{2n-1} - \vm^2 \rmm^{2n-1}) + \frac{2\rho^2}{c\kappa_0}(\rpp^{2n} - \rmm^{2n}) + 4\rho (1- \lambda^2) .
\end{aligned}
\end{equation*}
and 
\begin{equation*}
\begin{aligned}
 \sum_{j=1}^n \kappa_{n-j}^2 \zeta_j &=  (\lambda^2-1)\zeta_n -2\rho \mathcal{S}_1 + \frac{\vp^2}{\kappa_0} \mathcal{S}_2 - \frac{\vm^2}{\kappa_0} \mathcal{S}_3, 
\end{aligned}
\end{equation*}
where
\begin{equation*}
\begin{aligned}
\mathcal{S}_1 &= \sum_{j=1}^n j \kappa_{n-j}^2  
= \sum_{j=1}^n  j (\vp \rpp^{n-j} - \vm\rmm^{n-j})^2 \\
& = -\rho n(n+1)  + \vp^2 \rpp^{2n} \sum_{j=1}^n j\rmm^{2j} + \vm^2 \rmm^{2n} \sum_{j=1}^n j\rpp^{2j} \\
& =  -\rho n(n+1)  + \vp^2 (\rpp^{2n} - n-1+ n \rmm^2)/ \kappa_0^2 + \vm^2  (\rmm^{2n} -n-1 + n \rpp^2) / {\kappa_0^2} \\
& =  -\rho n(n+1)+ n(\lambda^2 - 1) - 1 + {(\vp^2\rpp^{2n} +\vm^2\rmm^{2n} - 2 \rho)}/{\kappa_0^2},
\end{aligned}
\end{equation*}
\begin{equation*}
\begin{aligned}
\mathcal{S}_2 &= \sum_{j=1}^n \kappa_{n-j}^2  \rpp^{2j-1} 
= \sum_{j=1}^n (\vp \rpp^{n-j} - \vm\rmm^{n-j})^2 \rpp^{2j-1}\\
& = n \vp^2 \rpp^{2n-1}  - 2\rho \sum_{j=1}^n \rpp^{2j-1} +  \vm^2\rmm^{2n+1} \sum_{j=1}^n \rpp^{4j} \\
& = n \vp^2 \rpp^{2n-1} - 2\rho(\rpp^{2n} -1)/\kappa_0 +  \vm^2 (\rpp^{2n+1} - \rmm^{2n-1} / (c\kappa_0).
\end{aligned}
\end{equation*}
\begin{equation*}
\begin{aligned}
\mathcal{S}_3 &= \sum_{j=1}^n \kappa_{n-j}^2  \rmm^{2j-1} 
= \sum_{j=1}^n (\vp \rpp^{n-j} - \vm\rmm^{n-j})^2 \rmm^{2j-1}\\
& = n \vm^2 \rmm^{2n-1}  - 2\rho \sum_{j=1}^n \rmm^{2j-1} +  \vp^2 \rpp^{2n+1} \sum_{j=1}^n \rmm^{4j} \\
& =  n \vm^2 \rmm^{2n-1} + 2\rho (\rmm^{2n} -1) / \kappa_0 - \vp^2 (\rmm^{2n+1} - \rpp^{2n-1}) / (c\kappa_0).
\end{aligned}
\end{equation*}
The result is then obtained by combining the terms together.
\end{proof}

\begin{corollary} \label{cor5}
If $c > 2$, $\lim_{n\rightarrow \infty} n^{-2}\tr(Q^{-2}) = 0$.
\end{corollary}
\begin{proof}
Note that $\kappa^2 = \vp^4 \rpp^{2n-2} -2\rho^2 + \vm^4 \rmm^{2n-2}$. If $c > 2$, then $\rpp > 1$, which implies $\lim_{n \rightarrow \infty} \rpp^n = \infty$ and hence $\lim_{n \rightarrow \infty} \kappa^2 = \infty$. Since $0 < \rmm < 1$:  $\lim_{n \rightarrow \infty} \rmm^n = 0$. From Theorem \ref{thm_trace2}, $\lim_{n\rightarrow \infty} n^{-2}\tr(Q^{-2}) = \kappa_0^{-2} \lim_{n\rightarrow \infty} \mathcal{S}/(n^2\kappa^2)$ where, $\lim_{n\rightarrow \infty} \mathcal{S}/(n^2\kappa^2)$ equals
\begin{equation*}
\begin{aligned}
&\lim_{n\rightarrow \infty} \left\{  \frac{4\rho^2}{\kappa^2} + \frac{4\rho(1 + \lambda^2) +2(\lambda^2 -1)^2}{n^2\kappa^2} 
+ \frac{c(1 - \vm^4\vp^{-4} \rmm^{4n-4})}{\kappa_0 n (1 - 2\rho^2\vp^{-4} \rmm^{2n-2} + \vm^4\vp^{-4} \rmm^{4n-4})}  
+ \frac{16\rho^2}{\kappa_0^2n^2 \kappa^2}
 \right.\\
& \left. \qquad 
-\frac{8\rho(\lambda^2-1)}{n\kappa^2} 
+ \frac{2\kappa_0(\lambda^2-1)(1 - \vm^2\vp^{-2} \rmm^{4n-2} ) - 4\rho c(1 +\vm^2 \vp^{-2}\rmm^{4n-2})}{\kappa_0^2 n^2 (\vp^2 \rmm -2\rho^2\vp^{-2}\rmm^{2n-1} + \vm^4\vp^{-2} \rmm^{4n-3})} \right\} = 0.
\end{aligned}
\end{equation*}
\\[-12mm]
\end{proof}

\section{Bounds for row sums of the inverse}
In this section, we focus on the AR(1) plus noise model which is discussed in Section \ref{sec:eg} and derive bounds for row sums of $Q^{-1}$. Suppose $|\phi| < 1$, $\phi \neq 0$ and $\gamma > 0$. Let $b = \phi/ |\phi|$, $c = (1+\gamma+\phi^2) /|\phi| > 2$ and $d = (1+\gamma)/ |\phi| > 1$. Note that $\lambda = |\phi| < d$. We first present Lemma \ref{boundv} before deriving the bound in Theorem \ref{thm_bds}.

\begin{lemma} \label{boundv}
For $i=1, \dots, n$, $|v_i + v_{n-i+1}|< 2/(d-\lambda)$.
\end{lemma}
\begin{proof}
First we show that $\kappa |v_i + v_{n-i+1}| \leq \kappa_{n-1} + \kappa_0$ for $i=1, \dots, n$. Taking into account Corollary \ref{cor_cnot2},
\begin{equation*}
\kappa|v_i + v_{n-i+1}| 
= {|b^{i-1}\kappa_{n-i} + b^{n-i} \kappa_{i-1}|}
\leq {\kappa_{n-i} + \kappa_{i-1}}.
\end{equation*}
Thus the desired inequality holds for $i=1$ by direct substitution. Applying the inequality in Property \ref{kappadiff} (repeatedly),
\begin{equation*}
\kappa_{n-i} + \kappa_{i-1} < {\kappa_{n-i+1} + \kappa_{i-2}} 
< \dots < {\kappa_{n-1} + \kappa_0},
\end{equation*}
for $i=2, \dots, \lceil n/2 \rceil$. Thus $\kappa |v_i + v_{n-i+1}| \leq \kappa_{n-1} + \kappa_0$ for $i=1, \dots, \lceil n/2 \rceil$. By symmetry, this inequality holds for $i=1, \dots, n$, since $v_i + v_{n-i+1} = v_{n-i+1} + v_{n-(n-i+1)+1}$. It suffices to show that $ (d-\lambda)(\kappa_{n-1} + \kappa_0) < 2\kappa$. Writing $\kappa = d \kappa_{n-1} - \kappa_{n-2}$, applying the identity of Property \ref{kappadiff} and using the fact that $\kappa_1 = d\kappa_0$,
\begin{equation*} \label{eq_ident1}
\begin{aligned}
2\kappa- (d-|\phi|)(\kappa_{n-1} + \kappa_0) 
& = c (\kappa_{n-1} - \kappa_{n-2}) + (c-2)\kappa_{n-2} -  \kappa_1 + |\phi| \kappa_0 \\
& = c \left( \kappa_1 - \kappa_0 + (c-2) \sum_{j=1}^{n-2} \kappa_j  \right)+ (c-2)\kappa_{n-2} - 2 \kappa_1 +c \kappa_0 \\
& = (c-2) \left( \kappa_1 + c \sum_{j=1}^{n-2} \kappa_j + \kappa_{n-2}  \right) > 0.
\end{aligned}
\end{equation*}
\end{proof}

\begin{theorem} \label{thm_bds}
If $0 < \phi <1$,
\begin{equation*}
 \frac{1}{d - \phi} <s_i < \frac{1}{c-2}.
\end{equation*}
If $-1 < \phi < 0$, 
\begin{equation*}
\frac{2}{c+2} - \frac{1}{d + \phi} < s_i <  \frac{1}{d + \phi}.
\end{equation*}
\end{theorem}
\begin{proof}
If $0 < \phi <1$, $b=1$ and $\{v_i|i=1, \dots, n$\} is positive from Property \ref{uvtve}. Thus, 
\begin{equation*}
s_i = \frac{ 1-(1-\phi)(v_i + v_{n-i+1}) }{c-2} < \frac{1}{c-2}.
\end{equation*}
This bound is tight as $\phi$ approaches 1. From Lemma \ref{boundv}, $v_i + v_{n-i+1} < 2/(d-\phi)$ implies that 
\begin{equation*}
s_i> \frac{1-2(1-\phi)/(d-\phi)}{c-2} =  \frac{1}{d - \phi}.
\end{equation*}
If $-1 < \phi < 0$, $b=-1$. From Lemma \ref{boundv}, $|v_i + v_{n-i+1}| < 2/(d+\phi)$ implies that 
\begin{equation*}
\begin{gathered}
\frac{1-2(1-\phi)/(d+\phi)}{c+2} < s_i = \frac{1+(1-\phi)(v_i + v_{n-i+1})}{c+2} < \frac{1+2(1-\phi)/(d+\phi)}{c+2}.
\end{gathered}
\end{equation*}
The result is obtained by simplifying the above inequality.
\end{proof}

\section{Application to AR(1) process with observational noise} \label{sec:eg}
Consider a univariate state space model where the observations $\{y_t\}$ are noisy and the latent states $\{x_t\}$ follow a stationary AR(1) process,
\begin{equation}\label{CP} 
\begin{aligned}
y_t &= x_t + \sigma_\epsilon \epsilon_t, \quad  (t = 1, \ldots, n),\\
x_{t} &= \mu + \phi (x_{t-1} - \mu) + \sigma_\eta \eta_t, \quad (t = 1, \ldots, n), \\
x_0 &\sim N\left(\mu, {\sigma_\eta^2}/{(1-\phi^2)} \right).
\end{aligned}
\end{equation} 
The $\{\epsilon_t\}$ and $\{\eta_s\}$ sequences are distributed as standard normals, and they are independent of each other and of $\{x_t\}$ for all $t$ and $s$. Let $\gamma = \sigma_\eta^2/\sigma_\epsilon^2$ be the signal-to-noise ratio and $\Theta = (\mu, \sigma_\eta, \sigma_\epsilon, \phi)^T$ be the vector of model parameters, where $|\phi| < 1$, $\sigma_\eta>0$, $\sigma_\epsilon > 0$ and $\mu \in \mathbb{R}$. This model is called the AR(1) plus noise or ARMA(1,1) model. \cite{Pitt1999} and \cite{Fruhwirth-Schnatter2004} observe that the convergence rate of the Gibbs sampler, when applied to the AR(1) plus noise model, is dependent on the parametrization. In particular, \eqref{CP} is known as the {\em centered} parametrization as the latent states are centered about $\mu$ and depend on $\sigma_\eta^2$. The {\em noncentered} parametrization is given by
\begin{equation}\label{NCP} 
\begin{aligned}
y_t &= \sigma_\eta b_t + \mu + \sigma_\epsilon \epsilon_t, \quad (t = 1, \ldots, n),\\
b_{t} &= \phi  b_{t-1} + \sigma_\eta \eta_t, \quad (t = 1, \ldots, n), \\
b_0 &\sim N\left(0, {1}/{(1-\phi^2)} \right),
\end{aligned}
\end{equation} 
where the latent states $\{b_t\}$ are now independent of $\mu$ and $\sigma_\eta^2$. These two models are equivalent as the likelihood, $p(y|\Theta)$, where $y = (y_1, \dots, y_n)^T$, obtained upon integrating out the latent states, is the same regardless of the parametrization. As the optimal parametrization is data dependent, \cite{Tan2017} introduced a partially noncentered parametrization,
\begin{equation}\label{PNCP} 
\begin{aligned}
y_t &= \sigma_\eta^a \alpha_t + w_t \mu + \sigma_\epsilon \epsilon_t, \quad  (t = 1, \ldots, n),\\
 \sigma_\eta^a \alpha_t &= \tilde{w}_t \mu + \phi (\sigma_\eta^a \alpha_{t-1} - \tilde{w}_t \mu) + \sigma_\eta \eta_t, \quad (t = 1, \ldots, n), \\
\sigma_\eta^a \alpha_0 &\sim N\left( \tilde{w}_0 \mu, {\sigma_\eta^2}/{(1-\phi^2)} \right),
\end{aligned}
\end{equation} 
where $\tilde{w}_t = 1 - w_t$.  Let $w = (w_1, \dots, w_n)^T$ and $\tilde{w} = \bone - w$, where $\bone$ is a vector of ones of length $n$. This parametrization includes the centered ($w=0$, $a=0$) and noncentered ($w=\bone, a=1$) parametrizations as special cases. Of interest are the values of the working parameters, $w$ and $a$, which optimize the convergence rates of the Gibbs sampler or the expectation-maximization (EM) algorithm, used to fit the AR(1) plus noise model to a time series. \cite{Tan2017} showed that the convergence rate of the EM algorithm is optimized by
\begin{equation} \label{wopt}
w^{\opt} = \bone - \sigma_\epsilon^{-2} \Omega^{-1} \bone,
\end{equation}
if $\phi$, $\sigma_\epsilon^2$, $\sigma_\eta^2$ are known, and by
\begin{equation} \label{aopt}
a^{\opt} = 1 - \frac{z^T \Omega^{-1} \Lambda \Omega^{-1}z}{n \sigma_\eta^2}, \quad w^{\opt} = (\mu\Omega)^{-1} \left( \frac{2\Lambda \Omega^{-1} z}{a^{\opt} \sigma_\eta^2} - z \right),
\end{equation}
if $\mu$, $\phi$, $\sigma_\epsilon^2$ are known. An alternating expectation conditional maximization algorithm, which uses these parametrizations in different cycles was then proposed for inferring $\Theta$. In \eqref{wopt}, $\Omega = \sigma_\eta^{-2} |\phi| Q$ is a tridiagonal matrix, where $Q$ is of the form in \eqref{tridiag_form} if $\phi \neq 0$, with $b= \phi/|\phi|$, $ c= (1+\gamma+\phi^2)/|\phi| > 2$, $d= (1+\gamma)/|\phi| > 1$ and $\lambda = c-d = |\phi|$. If $\phi=0$, $\Omega = \sigma_\eta^{-2} (1+\gamma)I$. In \eqref{aopt}, $z = \sigma_\epsilon^{-2}(y- \mu\bone)$ and $\Lambda$ is a tridiagonal matrix with diagonal $(1, 1+\phi^2, \dots, 1+\phi^2, 1)$ and off-diagonal elements equal to $-\phi$. 

To investigate the behavior of $w^\opt$ and $a^\opt$, such as their bounds, dependence on $\phi$ and $\gamma$ and large-sample properties, we require the explicit inverse of $Q$ and its properties. Suppose we are interested in understanding which parametrization is preferred when inferring $\mu$ given $\phi$, $\sigma_\epsilon^2$, $\sigma_\eta^2$. The expression of $w^{\opt}$ in \eqref{wopt} involves $\Omega^{-1}\bone$ which represents the row sums of $\Omega^{-1}$. Hence we can use Theorem \ref{thm_rowsum} to compute the elements in $w^{\opt}$. If $\phi \neq 0$, 
\begin{equation*}
\begin{aligned}
w_i^{\opt} &= 1 - \frac{ \gamma}{|\phi|}s_i  =\frac{(1-\phi)^2 +b\gamma (1-\phi) (v_i + v_{n-i+1})}{(1-\phi)^2 + \gamma}.
\end{aligned}
\end{equation*}
From this expression, we observe that $w^\opt$ is a centrosymmetric vector and $w_i^{\opt}$ depends only on $\phi$ and the signal-to-noise ratio $\gamma$. From Theorem \ref{thm_bds}, we can easily obtain bounds for $w_i^\opt$. If $0 < \phi < 1$,
\begin{equation*}
1 - \frac{\gamma}{(1-\phi)^2 + \gamma} <w_i^\opt < 1 - \frac{\gamma}{1- \phi^2 + \gamma }.
\end{equation*}
Thus $w_i^\opt$ lies strictly in the interval $(0,1)$ if $0 < \phi < 1$. Moreover, $w_i^{\opt} \rightarrow 0$ (centered parametrization is preferred) as $\phi \rightarrow 1$ or $\gamma \rightarrow \infty$. If $ -1 < \phi < 0$, then 
\begin{equation*}
 1 - \frac{\gamma}{1- \phi^2 + \gamma } <w_i^\opt < 1 + \frac{\gamma}{1- \phi^2 + \gamma } - \frac{2\gamma}{(1- \phi)^2 + \gamma }.
\end{equation*}
In this case, $w_i^{\opt}$ is positive but not necessarily bounded above by 1. Figure \ref{fig1} shows the values of $w_i^{\opt}$ and its bounds for some data simulated from the AR(1) plus noise model. We set $n=100$, $\mu=3$, $\sigma_\eta^2 = 0.02$, $\sigma_\epsilon^2 = 0.1$ and $\phi = \{-0.95, 0.1, 0.95\}$. As noted above, $w_i^{\opt}$ is close to zero when $\phi=0.95$, close to one when $\phi=0.1$ and it is not bounded above by one when $\phi <0$. Note that the signal-to-noise ratio, $\gamma=0.2$.
\begin{figure}[htb!]
\centering
\includegraphics[width=0.32\textwidth]{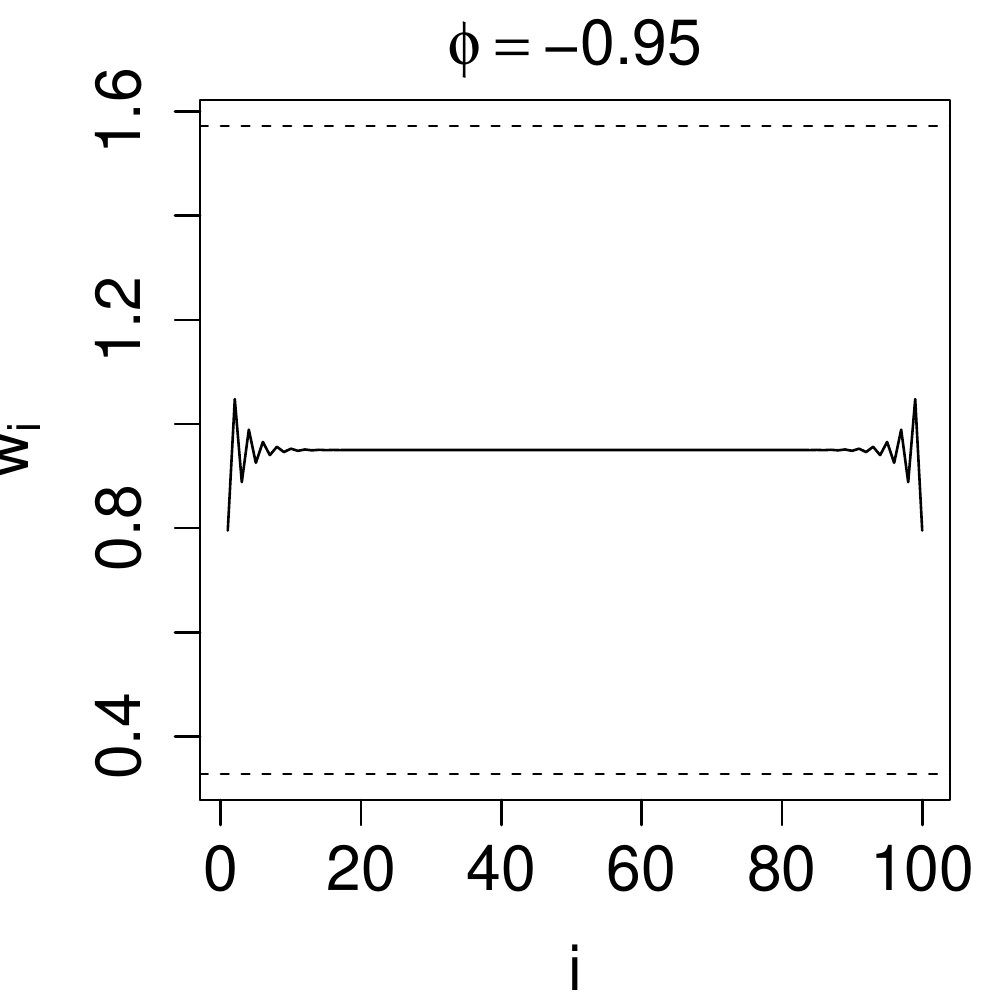}
\includegraphics[width=0.32\textwidth]{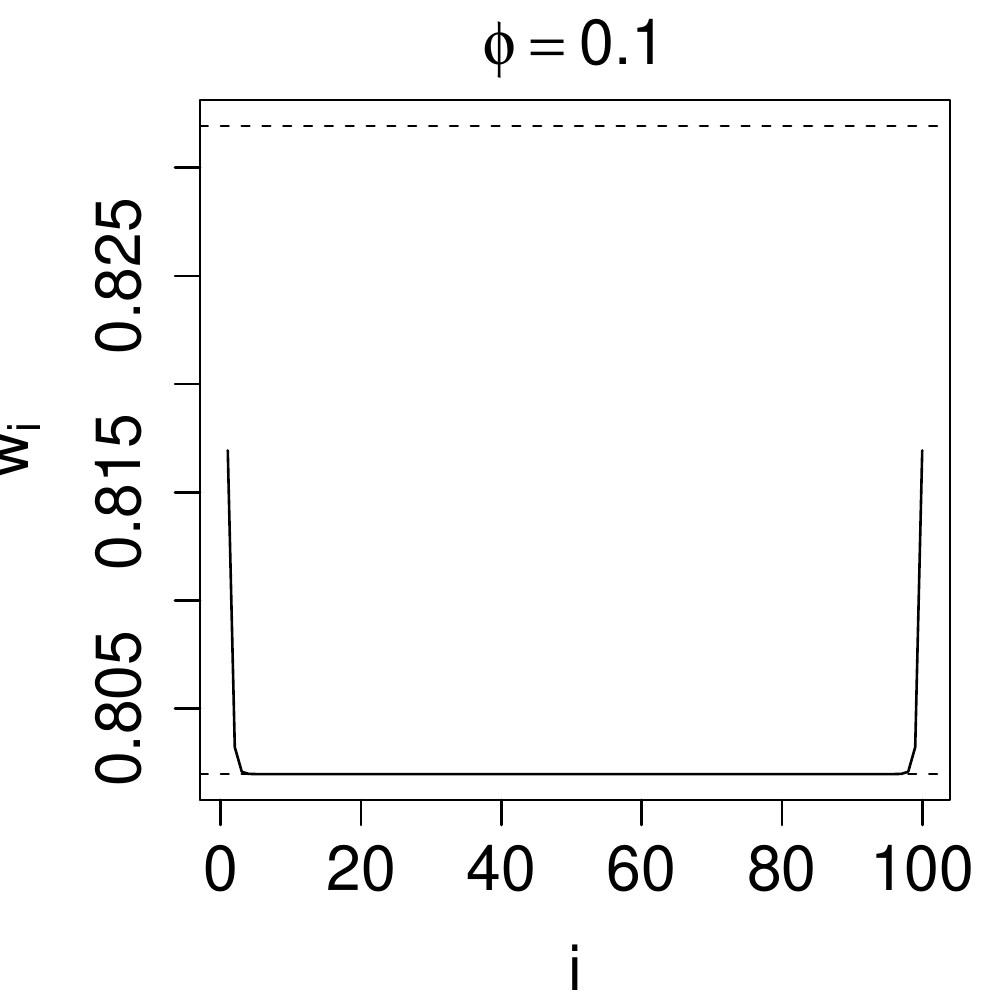}
\includegraphics[width=0.32\textwidth]{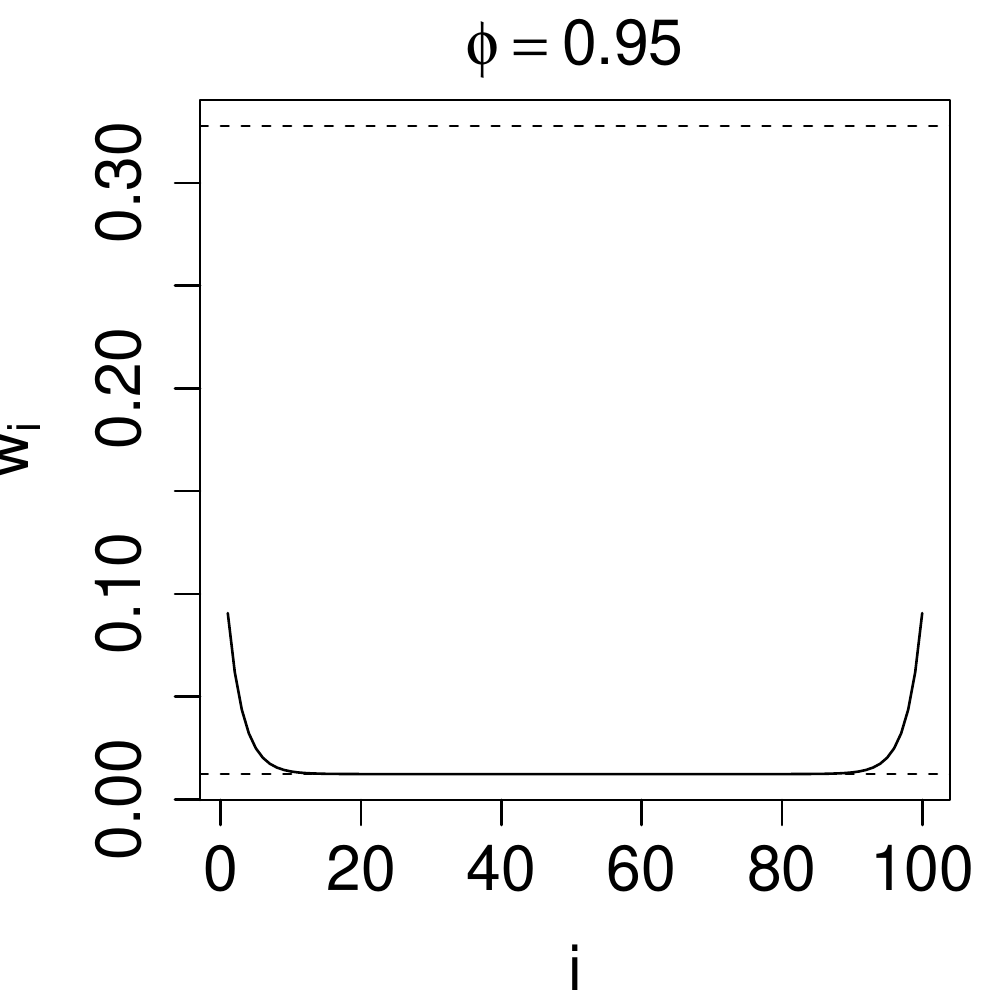}
\caption{Values of $w_i^{\opt}$ for $i=1, \dots, 100$ (solid) and the lower and upper bounds (dotted). \label{fig1}}
\end{figure}

Now consider the value of $a^{\opt}$ in \eqref{aopt}. The expression of $a^\opt$ is highly complex and it depends directly on the observations $y$ through $z$. As it is difficult to infer the behavior of $a^{\opt}$ directly from \eqref{aopt}, we attempt to study its large sample properties instead. \cite{Tan2017} showed that the mean and variance of $a^{\opt}$ are given by
\begin{equation*}
E(a^{\opt}) = 1 - \frac{\gamma\tr(Q^{-1})}{n|\phi|},\quad  \var(a^\opt) = \frac{2\gamma^2\tr(Q^{-2})}{n^2\phi^2}.
\end{equation*}
From Corollary \eqref{cor4}, 
\begin{equation*}
\begin{aligned}
\lim_{n \rightarrow \infty} E(a^\opt) &= 1 - \frac{\gamma}{|\phi|}  \lim_{n \rightarrow \infty} \frac{\tr(Q^{-1})}{n} 
& = 1 - \frac{\gamma}{|\phi|\kappa_0} = 1 - \frac{\gamma}{\sqrt{[(1-\phi)^2 + \gamma][(1+\phi)^2 + \gamma]}}. 
\end{aligned}
\end{equation*}
From Corollary \eqref{cor5}, $\lim_{n \rightarrow \infty} \var(a^\opt) = 2\gamma^2/(\phi^2) \lim_{n \rightarrow \infty} \tr(Q^{-2})/n^2 = 0$. Hence, we may consider using the limit of $E(a^\opt)$ as an estimate of $a^{\opt}$ when the sample size $n$ is large as this limit is more efficient to compute than $a^\opt$. Figure \ref{fig2} shows how the values of $E(a^\opt)$ and $\var(a^\opt)$ approach their limits as $n$ increases when $\phi = 0.95$ and $\gamma = 0.2$. The values of $E(a^\opt)$ and $\var(a^\opt)$ are computed using Theorems \ref{thm_trace} and \ref{thm_trace2}, which provide efficient ways to compute the trace of the inverse matrices when $n$ is large.
\begin{figure}[htb!]
\centering
\includegraphics[width=0.75\textwidth]{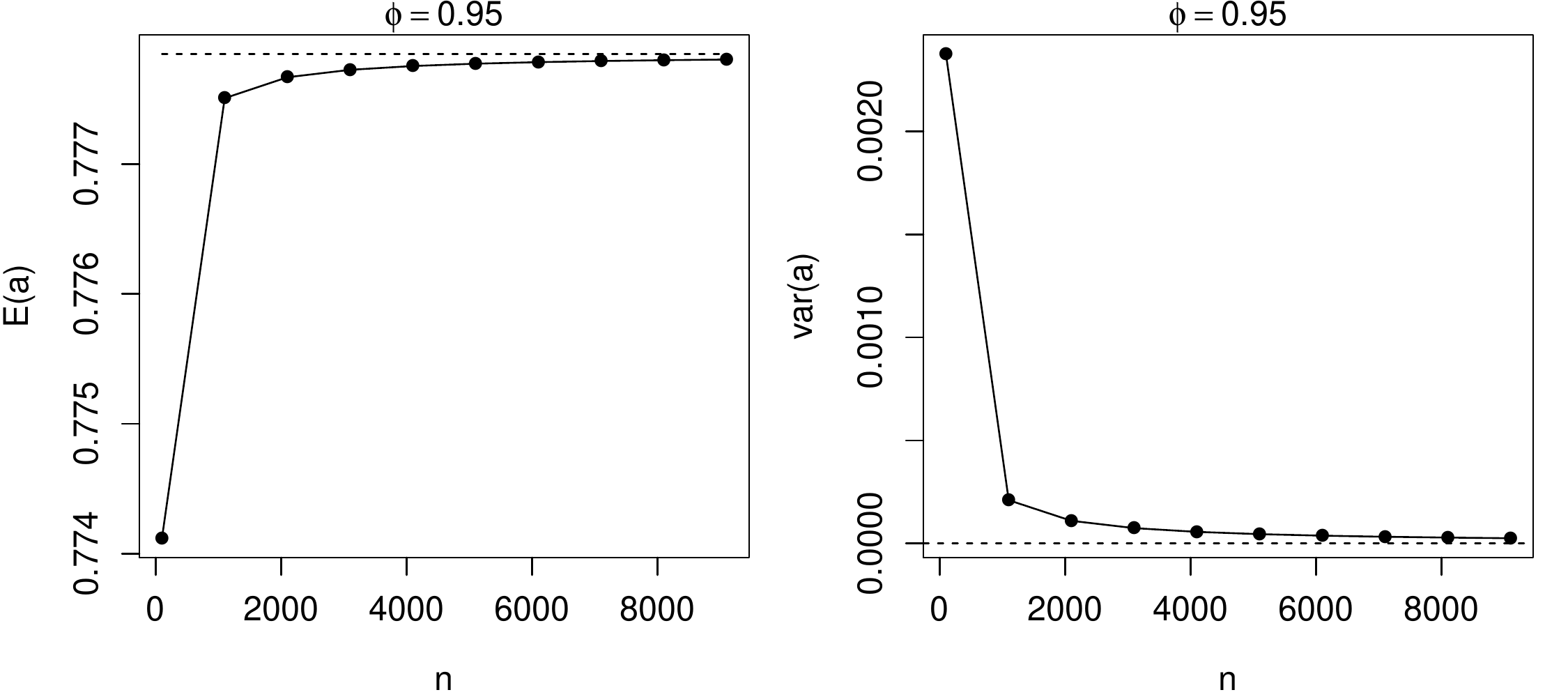}
\caption{Values of $E(a^\opt)$ and $\var(a^\opt)$ for $n=1000, 2000, \dots, 10000$. The limits are shown in dotted lines. \label{fig2}}
\end{figure}

\section{Conclusion}
In this article, we have derived explicit expressions for the inverse of a class of tridiagonal matrices that often arise in interpolation problems and in statistical models which use first order autoregression to induce dependence in the covariance structure. Such analytic formulas will be important in model estimation and for studying the properties of estimators, such as their bounds or large-sample behavior as illustrated in the application to the AR(1) plus noise model.

\section{Acknowledgments}
Linda Tan is supported by the start-up grant (R-155-000-190-133). We thank the editor, associate editor and referees for their comments which have helped to improve the manuscript.

\bibliographystyle{chicago}
\bibliography{ref}

\end{document}